\documentclass[12pt,reqno]{amsart}

\usepackage[latin1]{inputenc}
\usepackage{amsmath}
\usepackage{amsfonts}
\usepackage{amssymb}
\usepackage{graphics}
\usepackage{enumerate}
\usepackage{amssymb,amsmath,amsthm,amscd,latexsym,verbatim,graphicx,amsfonts}

\usepackage{amscd}
\usepackage{amsmath}
\usepackage{amssymb}
\usepackage{amsthm}
\usepackage{latexsym}
\usepackage{verbatim}

\theoremstyle{plain}
\newtheorem{theorem}{Theorem}[section]

\newtheorem{corollary}[theorem]{Corollary}
\newtheorem{lemma}[theorem]{Lemma}
\newtheorem{prop}[theorem]{Proposition}
\theoremstyle{definition}

\theoremstyle{remark}

\newcommand{\nri}{n\rightarrow\infty}

\newcommand{\bbR}{\mathbb{R}}
\newcommand{\bbC}{\mathbb{C}}

\newcommand{\bbD}{\mathbb{D}}
\newcommand{\bbJ}{\mathbb{J}}
\newcommand{\bbN}{\mathbb{N}}

\newcommand{\sinc}{\mathrm{sinc}}
\newcommand{\kri}{k\rightarrow\infty}
\newcommand{\eitheta}{e^{i\theta}}

\DeclareMathOperator*{\supp}{supp}
\DeclareMathOperator*{\Real}{Re}

\title[]{Applications of a New Formula for OPUC with Periodic Verblunsky Coefficients}
\author[]{Brian Simanek}
\date{}

\begin{document}
\maketitle

\begin{abstract}
We find a new formula for the orthonormal polynomials corresponding to a measure $\mu$ on the unit circle whose Verblunsky coefficients are periodic.  The formula is presented using the Chebyshev polynomials of the second kind and the discriminant of the periodic sequence.  We present several applications including a resolution of a problem suggested by Simon in 2006 regarding the existence of singular points in the bands of the support of the measure and a universality result at all points of the essential support of $\mu$.
\end{abstract}

\vspace{4mm}


\vspace{2mm}


\vspace{2mm}

\normalsize

\section{Introduction}\label{intro}

Let $\mu$ be a probability measure on the unit circle with infinite and compact support.  If $\{\Phi_n\}_{n=0}^{\infty}$ denotes the sequence of monic orthogonal polynomials obtained by applying Gram-Schmidt orthogonalization to the sequence $\{z^n\}_{n=0}^{\infty}$, then we define the sequence of Verblunsky coefficients $\{\alpha_n\}_{n=0}^{\infty}$ by $-\bar{\alpha}_n=\Phi_{n+1}(0)$.  The well-known Szeg\H{o}-recursion states
\begin{align*}
\Phi_{n+1}(z)&=z\Phi_n(z)-\bar{\alpha}_{n}\Phi_n^*(z)\\
\Phi_{n+1}^*(z)&=\Phi_n^*(z)-\alpha_nz\Phi_n(z),
\end{align*}
where $\Phi_n^*(z)=z^n\overline{\Phi_n(1/\bar{z})}$.  The second kind polynomials $\{\Psi_n(z)\}_{n=0}^{\infty}$ are obtained by iterating the Szeg\H{o} recursion using the sequence $\{-\alpha_n\}_{n=0}^{\infty}$ and the initial condition $\Psi_0(z)=1$.  The second kind polynomials are also orthogonal with respect to a different measure on the unit circle, which is in the family of Aleksandrov measures for the measure $\mu$.  Let $\{\varphi_n\}_{n=0}^{\infty}$ denote the sequence of orthonormal polynomials for the measure $\mu$ and let $\{\psi_n\}_{n=0}^{\infty}$ denote the sequence of normalized second kind polynomials.

Verblunsky's Theorem (see \cite[Section 1.7]{OPUC1}) establishes a bijection between sequences $\{\alpha_n\}_{n=0}^{\infty}$ in $\bbD$ and probability measures on $\partial\bbD=\{z:|z|=1\}$ with infinite support.  Therefore it is meaningful to study measures $\mu$ whose corresponding sequence satisfies certain properties.  In this paper, we will focus on the case when the sequence $\{\alpha_n\}_{n=0}^{\infty}$ is periodic with period $p$.  Polynomials $\{\Phi_n\}_{n=0}^{\infty}$ generated by periodic sequences $\{\alpha_n\}_{n=0}^{\infty}$ have been studied extensively and much is known about their structure, asymptotics, and zero behavior (see for example  \cite{Luk, PehSt,PehSt97, PehSt2,PehSt2000,PehSt2000a,Simon3}).  However, we will introduce a new formula for these poiynomials that allows us to carry out some complicated calculations and produce several new results.  The most substantial of our results are a demonstration of the existence of singular points (defined below) in the support of the measure $\mu$ and also a universality result for the scaled polynomial reproducing kernels around any point in the essential support of $\mu$.

A fundamental tool in the analysis of polynomials generated by periodic sequences of Verblunsky coefficients is the function $\Delta(z)$ defined by
\begin{align}\label{deltadef}
\Delta(z)&:=\frac{\varphi_p(z)+\varphi_p^*(z)+\psi_p(z)+\psi_p^*(z)}{2z^{p/2}},
\end{align}
which we call the discriminant (as in \cite[Section 11]{OPUC2}).  Due to the presence of the factor $z^{p/2}$, it will sometimes be convenient to assume that $p$ is even.  Note that there is no loss of generality in making this assumption because every sequence of period $p$ is also a sequence of period $2p$.  The discriminant is especially useful in describing the support of the corresponding measure.  By \cite[Theorems 11.1.1 and 11.1.2]{OPUC2}, we know that if $\mu$ is a measure on the unit circle whose Verblunsky coefficients are periodic with period $p$, then the support of the measure $\mu$ consists of disjoint closed arcs $\{B_k\}_{k=1}^q$ (where $q\leq p$) and at most one point in each gap between the arcs.  These arcs (called \textit{bands}) satisfy the relation
\[
\bigcup_{k=1}^qB_k=\{\eitheta:|\Delta(\eitheta)|\leq2\}
\]
On each band, the measure $\mu$ is purely absolutely continuous with respect to arclength measure.  The set $\partial\bbD\setminus\cup B_k$ is called the set of \textit{gaps} in $\supp(\mu_{ac})$.

Let us denote the endpoints of $B_k$ by $e^{ix_k}$ and $e^{iy_k}$, where $0\leq x_1<y_1<x_2<y_2<\cdots<y_q<x_1+2\pi$.  When $p$ is even, it holds that
\[
\{e^{ix_j},e^{iy_j}:j=1,\ldots,q\}\subseteq\Delta^{-1}(\{-2,2\}),
\]
but this inclusion could be strict.  Any point $\eitheta$ for which $|\Delta(\eitheta)|=2$ but $\eitheta$ is not the endpoint of a band is called a \textit{closed gap}.

We can also use $\Delta$ to derive a formula for the equilibrium measure of $\supp(\mu)$.  Indeed, it is known that the equilibrium measure is given by\footnote{Here and always, we identify the unit circle with the interval $[0,2\pi)$.}
\begin{equation}\label{vdef}
V(\theta)\frac{d\theta}{2\pi}=\frac{|W(\theta)|}{p\sqrt{4-\Delta(\eitheta)^2}}\frac{d\theta}{2\pi},
\end{equation}
where
\begin{equation}\label{wdef}
W(s):=\frac{d}{ds}\Delta(e^{is})
\end{equation}
(see \cite[Theorem 11.1.3]{OPUC2} and Remark 1 following it).  These formulas for $V$ and $W$ will be important when stating our results in Section \ref{universality}.

As in \cite{SimGeron}, our results will be made possible by a new formula for the orthonormal polynomials $\{\varphi_n\}_{n=0}^{\infty}$ in terms of the Chebyshev polynomials of the second kind.  These polynomials have many closed form expressions, but the one that will be most useful for us is

\begin{equation}\label{chebform1}
U_n(x)=\sum_{j=0}^{\left\lfloor\frac{n}{2}\right\rfloor}(-1)^j\binom{n-j}{j}(2x)^{n-2k}
\end{equation}
(see \cite[page 37]{BE}).  We also recall from \cite[page 37]{BE} the formulas
\begin{equation}\label{chebform2}
U_n(\cos(t))=\frac{\sin((n+1)t)}{\sin(t)}
\end{equation}
\begin{equation}\label{chebform3}
U_n(x)=\frac{(x+\sqrt{x^2-1})^{n+1}-(x-\sqrt{x^2-1})^{n+1}}{2\sqrt{x^2-1}}
\end{equation}
The detailed knowledge we have about the Chebyshev polynomials will enable us to deduce equally precise results for the orthonormal polynomials when the Verblunsky coefficients are periodic.

\smallskip

In the next section, we will present our new formula for $\varphi_n$ in Theorem \ref{closed}.  The proof is short and is similar to that of a related formula that appears in \cite{SimGeron}.  After deriving these formulas, the remainder of the paper is devoted to applications.

\section{Periodic Verblunsky Coefficients}\label{periodic}

Let us fix $p\in\bbN$ and consider a periodic sequence of Verblunsky coefficients $\{\alpha_n\}_{n=0}^{\infty}$ that satisfies $\alpha_{n+p}=\alpha_n$ for all $n\geq0$.  Let $\{\varphi_n\}_{n=0}^{\infty}$ be the corresponding sequence of orthonormal polynomials and let $\{\psi_n\}_{n=0}^{\infty}$ be the corresponding sequence of normalized second kind polynomials.  For convenience, we define
\begin{align*}
r&:=\prod_{j=0}^{p-1}\sqrt{1-|\alpha_j|^2}\\
\eta(z;\sigma)&:=\sigma(\varphi_p(z)-\varphi^*_p(z))-\psi_p(z)-\psi_p^*(z),
\end{align*}
which we will retain for the remainder of this paper.  Our first result will make all of our subsequent results possible and is given by the following theorem.

\begin{theorem}\label{closed}
For any $k\in\bbN_0$ and $s\in\{1,\ldots,p-1\}$, it holds that
\begin{align*}
\varphi_{kp}(z)&=z^{kp/2}\bigg[U_k\left(\frac{\Delta(z)}{2}\right)+\frac{\eta(z;1)}{2z^{p/2}}U_{k-1}\left(\frac{\Delta(z)}{2}\right)\bigg]\\
\varphi_{kp}^*(z)&=z^{kp/2}\bigg[U_k\left(\frac{\Delta(z)}{2}\right)+\frac{\eta(z;-1)}{2z^{p/2}}U_{k-1}\left(\frac{\Delta(z)}{2}\right)\bigg]\\
\varphi_{kp+s}(z)&=\frac{(\varphi_{s}(z)+\psi_s(z))\varphi_{kp}(z)+(\varphi_s(z)-\psi_s(z))\varphi_{kp}^*(z)}{2}\\
\varphi_{kp+s}^*(z)&=\frac{(\varphi_s^*(z)-\psi_s^*(z))\varphi_{kp}(z)+(\varphi_s^*(z)+\psi_s^*(z))\varphi_{kp}^*(z)}{2}
\end{align*}
where $U_{-1}=0$.
\end{theorem}

\noindent\textit{Remark.}  It is likely that these formulas can be obtained by combining \cite[Theorem 3.2]{PehSt} and \cite[Theorem 4.2]{PehSt}, but we provide a new short proof for the reader's convenience.

\begin{proof}
We will mimic the proof of \cite[Theorem 2.1]{SimGeron}.  If we set
\[
A_n=\begin{pmatrix}
z & -\bar{\alpha}_n\\
-\alpha_n z & 1
\end{pmatrix}
\]
then we can write the Szeg\H{o} recursion as
\[
\begin{pmatrix}
\Phi_{n+1}(z)\\
\Phi_{n+1}^*(z)
\end{pmatrix}=A_n
\begin{pmatrix}
\Phi_{n}(z)\\
\Phi_{n}^*(z)
\end{pmatrix}=A_nA_{n-1}\cdots A_0
\begin{pmatrix}
1\\
1
\end{pmatrix}
\]
Since the Verblunsky coefficients form a periodic sequence, we have (see \cite[Eq. 3.2.17]{OPUC1})
\[
\begin{pmatrix}
\Phi_{kp+s}(z)\\ \Phi_{kp+s}^*(z)
\end{pmatrix}=\frac{1}{2}
\begin{pmatrix}
\Phi_s(z)+\Psi_s(z) & \Phi_s(z)-\Psi_s(z)\\
\Phi_s^*(z)-\Psi_s^*(z) & \Phi_s^*(z)+\Psi_s^*(z)
\end{pmatrix}
\begin{pmatrix}
\frac{\Phi_p(z)+\Psi_p(z)}{2} & \frac{\Phi_p(z)-\Psi_p(z)}{2}\\
\frac{\Phi_p^*(z)-\Psi_p^*(z)}{2} & \frac{\Phi_p^*(z)+\Psi_p^*(z)}{2}
\end{pmatrix}^k
\begin{pmatrix}
1\\ 1
\end{pmatrix}
\]
(compare with \cite[Equation 3.22]{Simon3}).  Therefore, one can apply \cite[Theorem 1]{McL} with 
\[
y_k(z)=\sum_{m=0}^{\left\lfloor\frac{k}{2}\right\rfloor}\binom{n-m}{m}\left(z^{p/2}r\Delta(z)\right)^{k-2m}(-z^pr^2)^m=z^{pk/2}r^kU_k\left(\frac{\Delta(z)}{2}\right)
\]
to calculate the $k^{th}$ power of the rightmost matrix.  The desired result follows.
\end{proof}

\noindent\textit{Remark.}  The analog of Theorem \ref{closed} on the real line was proven by de Jesus and Petronilho (see \cite[Theorem 5.1]{dJP} and also \cite{Ger40}).  The special case $p=1$ was considered in \cite{SimGeron}.

\smallskip

\noindent\textit{Remark.}  Other formulas resembling those in Theorem \ref{closed} can be found in \cite[Theorem 2.1]{PehSt97}.

\smallskip

There are many simple consequences of the formulas in Theorem \ref{closed} and we explore many of them in Section \ref{simple}.  For now we turn to a more substantial analysis to resolve some unsolved problems.

\subsection{Singular Points in the Bands}\label{sing}

Notice that we can use Theorem \ref{closed} to deduce the Szeg\H{o} asymptotics of the orthonormal polynomials.  This has been done already in \cite{PehSt97,PehSt2,Simon3}, but we can provide a simple proof that comes with explicit bounds on the error terms.  When $\mu$ has periodic Verblunsky coefficients with period $p$ and $p$ is even, $\cup B_k=\Delta^{-1}([-2,2])$ (see \cite[Theorem 11.1.1]{OPUC2}) so we can define the functions $\Gamma_{\pm}$ as in \cite{Simon3} by
\begin{equation}\label{gamdef}
\Gamma_{\pm}(z)=\frac{\Delta(z)}{2}\pm\sqrt{\frac{\Delta(z)^2}{4}-1}
\end{equation}
Since $p$ is even, these functions are analytic on $\bbC\setminus\{\cup B_k\}$.  The formula \eqref{chebform3} implies
\[
U_n\left(\frac{\Delta(z)}{2}\right)=\frac{\Gamma_+(z)^{n+1}-\Gamma_-(z)^{n+1}}{\sqrt{\Delta(z)^2-4}}.
\]
The following result is now an immediate consequence of Theorem \ref{closed} and the formula (\ref{chebform3}).

\begin{theorem}\label{jdef}
Let $\{\varphi_n(z)\}_{n\geq0}$ be the sequence of orthonormal polynomials for the measure $\mu$, whose corresponding sequence of Verblunsky coefficients is periodic with period $p$, where $p$ is even.  For $z$ not in the bands of the support of $\mu$ it holds that
\begin{align*}
\lim_{k\rightarrow\infty}\frac{z^{-kp/2}\varphi_{kp}(z)}{\Gamma_+(z)^k}&=:j_0(z)=\frac{\Gamma_+(z)}{2\sqrt{\Delta(z)^2/4-1}}+\frac{\eta(z;1)}{4z^{p/2}\sqrt{\Delta(z)^2/4-1}}\\
\lim_{k\rightarrow\infty}\frac{z^{-kp/2}\varphi_{kp}^*(z)}{\Gamma_+(z)^k}&=:\ell_0(z)=\frac{\Gamma_+(z)}{2\sqrt{\Delta(z)^2/4-1}}+\frac{\eta(z;-1)}{4z^{p/2}\sqrt{\Delta(z)^2/4-1}}\\
\lim_{k\rightarrow\infty}\frac{z^{-kp/2}\varphi_{kp+s}(z)}{\Gamma_+(z)^k}&=:j_s(z)=\frac{(\varphi_s(z)+\psi_s(z))j_0(z)+(\varphi_s(z)-\psi_s(z))\ell_0(z)}{2},\\
\lim_{k\rightarrow\infty}\frac{z^{-kp/2}\varphi_{kp+s}^*(z)}{\Gamma_+(z)^k}&=:\ell_s(z)=\frac{(\varphi_s^*(z)-\psi_s^*(z))j_0(z)+(\varphi_s^*(z)+\psi_s^*(z))\ell_0(z)}{2},
\end{align*}
for each $s\in\{1,\ldots,p-1\}$.  Furthermore, in all cases the convergence is exponentially fast, i.e. for every compact set $X\subseteq\bbC\setminus\supp(\mu_{ac})$, there are constants $C_X>0$ and $\epsilon_X\in(0,1)$ so that
\begin{align*}
\left|\frac{z^{-kp/2}\varphi_{kp+s}(z)}{\Gamma_+(z)^k}-j_s(z)\right|\leq C_X(1-\epsilon_X)^k,\qquad\qquad z\in X\\
\left|\frac{z^{-kp/2}\varphi_{kp+s}^*(z)}{\Gamma_+(z)^k}-\ell_s(z)\right|\leq C_X(1-\epsilon_X)^k,\qquad\qquad z\in X
\end{align*}
for each $s\in\{0,\ldots,p-1\}$.
\end{theorem}

\begin{proof}
As mentioned above, the form of $j_s$ and $\ell_s$ is a direct consequence of Theorem \ref{closed} and the formula (\ref{chebform3}).  The statement about exponential convergence follows from the fact that $\Gamma_+$ maps the complement of the bands of $\supp(\mu)$ to the complement of the closed unit disk and $\Gamma_-$ maps the complement of the bands of $\supp(\mu)$ to the unit disk.
\end{proof}


With Theorem \ref{jdef} in hand we can now resolve an unsolved problem from \cite{Simon3} about the zeros of $\{j_s\}_{s=0}^{p-1}$.  In \cite{Simon3}, Simon defines singular points of order $k$ to be those points in the bands where $j_s$ vanishes to order $k$ (note that $j_s(\eitheta)$ is defined when $\eitheta$ is in a band as $\lim_{r\rightarrow1^-}j_s(r\eitheta)$).  His motivation for this definition comes from a desire to understand the asymptotic distribution of the zeros of the polynomial $\varphi_n$ as $\nri$.  Certain estimates can be made more precise in the absence of singular points.  It was also mentioned in \cite{Simon3} that there are no known examples with singular points.  We will show that there are indeed measures with periodic Verblunsky coefficients that have singular points.

Let $\nu$ be the equilibrium measure of $\supp(\mu)$.  As in \cite{Simon3}, define the function $k(\theta)$ to be the cumulative distribution function of $\nu$ 
\[
k(\theta)=\nu\left(\{e^{it}:x_1\leq t\leq\theta\}\right),
\]
where $x_1$ is as described in Section \ref{intro}.  If $s\in\{0,\ldots,n-1\}$, then simple algebraic manipulation of the formulas in Theorem \ref{jdef} and an application of \cite[Proposition 3.4]{Simon3} shows that $j_s(\eitheta)=0$ with $\eitheta$ in some band of the support of $\mu$ if and only if
\begin{align}\label{js}
2e^{-i\pi pk(\theta)}=-\eta\left(\eitheta;\frac{\psi_s(\eitheta)}{\varphi_s(\eitheta)}\right)
\end{align}

From (\ref{js}) we can deduce that there are in fact singular points in some cases.  Indeed, let us consider the sequence of Verblunsky coefficients $\{0,\ldots,0,\alpha,0,\ldots,0,\alpha,0,\ldots\}$, where there are $p-1$ zeros between each $\alpha$ ($\alpha$ will be chosen later).  Then $\psi_s/\varphi_s\equiv1$ for any $s\in\{0,\ldots,p-1\}$ so (\ref{js}) becomes
\begin{align}\label{jseq}
e^{-i\pi pk(\theta)}=\frac{1+\bar{\alpha}}{\sqrt{1-|\alpha|^2}}
\end{align}
for all $s\in\{0,1,\ldots,p-1\}$.  If we specify that $\Real[\alpha]=-|\alpha|^2\neq0$, then the right-hand side of (\ref{jseq}) has absolute value $1$.  Since the left-hand side of (\ref{jseq}) takes on all values in the unit circle $p/2$ times as $\theta$ runs from $x_1$ to $2\pi+x_1$, we see that we have exactly $p/2$ singular points and they are common to all $\{j_s\}_{s=0}^{p-1}$.  In fact, one can see that in this case there are precisely $p$ bands and they alternate between having a singular point and not having a singular point.

We note that the circle defined by the condition $\Real[\alpha]=-|\alpha|^2$ is precisely the boundary circle of the closed set $\Omega$ such that the measure of orthogonality has a mass point in each gap if and only if $\alpha\not\in\Omega$ (see \cite[Section 1.6]{OPUC1}).

\section{Universality}\label{universality}

Recall that the reproducing kernel for polynomials of degree at most $n$ in $L^2(\mu)$ is given by
\[
K_n(z,w;\mu):=\sum_{j=0}^n\varphi_j(z)\overline{\varphi_j(w)}
\]
Let $\sigma_n$ be a real sequence that monotonically tends to $0$ as $\nri$.  Our goal in this section is to understand the limit
\[
\lim_{\nri}\frac{K_n(e^{i(\theta+a\sigma_n)},e^{i(\theta+b\sigma_n)};\mu)}{K_n(\eitheta,\eitheta;\mu)}
\]
for appropriate values of $\theta\in\bbR$ and an appropriate sequence $\sigma_n$.  The two sequences we will be most interested in are $\sigma_n=-n^{-2}$ and $\sigma_n=n^{-1}$, depending on location of $\eitheta$ in the support of $\mu$.  A desire to study limits of the above form comes from random matrix theory, where such limits can be useful in calculating eigenvalue correlation functions.  There are many examples of measures on the unit circle for which these limits can be calculated (see \cite{Bourgade,LeviLub,LubNg,SimUniv,SimGeron}), but the case of a measure supported on several arcs of the unit circle has not been previously considered.  That is the case we aim to consider here and we note that the analogous results on the real line have been proven by Totik in \cite{Totik} and by Danka in \cite{Danka}.

We will consider measures $\mu$ on the unit circle whose corresponding sequence of Verblunsky coefficients is periodic with period $p$, where $p$ is even.  Our analysis will follow the general method pioneered by Lubinsky, which has been utilized repeatedly by other authors.  The main technical obstacle that we will need to overcome is to take special consideration of the set of resonances, which are the zeros of $\varphi_p(z)-\varphi_p^*(z)$ (see \cite[Section 11.3]{OPUC2}).  These zeros are all on the unit circle and are located in the closure of $\partial\bbD\setminus\Delta^{-1}((-2,2))$.  Also, every closed gap is a resonance  and it is possible for the endpoint of a band to be a resonance (see \cite[Theorem 11.3.1]{OPUC2}).

Before we can state our result, we need to define the kernel that will appear in our theorem.  For any $s\in\bbR$, let $J_s$ denote the Bessel function of the first kind order $s$.  Define
\[
\bbJ_s^*(a,b):=\begin{cases}
\frac{J_s(\sqrt{a})\sqrt{b}J_s'(\sqrt{b})-J_s(\sqrt{b})\sqrt{a}J_s'(\sqrt{a})}{2a^{s/2}b^{s/2}(a-b)}\qquad & a\neq b\\
\frac{1}{4a^s}\left(J_s(\sqrt{a})^2-J_{s+1}(\sqrt{a})J_{s-1}(\sqrt{a})\right) & a=b.
\end{cases}
\]
As noted in \cite{LubEdge}, $\bbJ_s^*(a,b)$ is entire on $\bbC^2$.  We also define $\sinc(x)=\frac{\sin(x)}{x}$.

Now we can state our main result of this section (recall the definition of $V$ from \eqref{vdef} and $W$ from \eqref{wdef}).

\begin{theorem}\label{univthm}
Let $\mu^*$ be a measure of the form $w(t)\mu+\mu_0$, where $\mu$ has periodic Verblunsky coefficients with period $p$, where $p$ is even.  Suppose that $w$ is positive and continuous everywhere on $\supp(\mu)$ and $\supp(\mu_0)\subseteq\supp(\mu_{ac})$.
\begin{itemize}
\item[i)]  If $\eitheta$ is an interior point of a band in $\supp(\mu)$ and $\mathrm{dist}(\eitheta,\supp(\mu_0))>0$, then
\[
\lim_{\kri}\frac{K_{n}\left(e^{i(\theta+a/n)},e^{i(\theta+b/n)};\mu^*\right)}{K_{n}(\eitheta,\eitheta;\mu^*)}=e^{i\frac{a-\bar{b}}{2}}\sinc(V(\theta)(a-\bar{b}))
\]
where the convergence is uniform for $a,b$ in compact subsets of $\bbC^2$.
\item[ii)]  If $\eitheta$ is an endpoint of a band in $\supp(\mu)$, $\mathrm{dist}(\eitheta,\supp(\mu_0))>0$, $\Delta(\eitheta)=2$, and $\eitheta$ is not a resonance, then
\[
\lim_{\kri}\frac{K_{n}\left(e^{i(\theta-a/n^2)},e^{i(\theta-b/n^2)};\mu^*\right)}{K_{n}(\eitheta,\eitheta;\mu^*)}=\frac{\bbJ_{1/2}^*\left(\frac{W(\theta)a}{p^2},\frac{W(\theta)\bar{b}}{p^2}\right)}{\bbJ_{1/2}^*\left(0,0\right)}
\]
where the convergence is uniform for $a,b$ in compact subsets of $\bbC^2$.
\item[iii)]  If $\eitheta$ is an endpoint of a band in $\supp(\mu)$, $\mathrm{dist}(\eitheta,\supp(\mu_0))>0$, $\Delta(\eitheta)=2$, and $\eitheta$ is a resonance, then
\[
\lim_{\kri}\frac{K_{n}\left(e^{i(\theta-a/n^2)},e^{i(\theta-b/n^2)};\mu^*\right)}{K_{n}(\eitheta,\eitheta;\mu^*)}=\frac{\bbJ_{-1/2}^*\left(\frac{W(\theta)a}{p^2},\frac{W(\theta)\bar{b}}{p^2}\right)}{\bbJ_{-1/2}^*\left(0,0\right)}
\]
where the convergence is uniform for $a,b$ in compact subsets of $\bbC^2$.
\end{itemize}
\end{theorem}

\noindent\textit{Remark.}  In parts (ii) and (iii), there exists a corresponding result for $\Delta(\eitheta)=-2$, but we will omit those calculations.

\smallskip

\noindent\textit{Remark.}  Our result allows for the possibility of mass points in the gaps between the bands of $\supp(\mu)$.

\smallskip

\noindent\textit{Remark.}  We note here that our assumptions imply $\mu^*$ is regular on $\supp(\mu)$ in the sense of Stahl and Totik (see \cite{StaTo}).

\smallskip

For the remainder of this section, let $\mu$ be fixed as in the statement of Theorem \ref{univthm}.  Our first task is to prove Theorem \ref{univthm} for $\mu^*=\mu$.  We will proceed to first understand the limit
\[
\lim_{\kri}\frac{K_{kp-1}(e^{i(\theta+a\sigma_{kp-1})},e^{i(\theta+b\sigma_{kp-1})};\mu)}{K_{kp-1}(\eitheta,\eitheta;\mu)}
\]
For brevity, let us define $z_a=e^{i(\theta+a\sigma_{kp-1})}$ and $z_b=e^{i(\theta+b\sigma_{kp-1})}$ for complex numbers $a$ and $b$.  If we assume $a\neq\bar{b}$, then we can use the Christoffel-Darboux formula (see \cite{SimonCD}) to write
\[
K_{kp-1}(z_a,z_b;\mu)=\frac{\varphi_{kp}^*(z_a)\overline{\varphi_{kp}^*(z_b)}-\varphi_{kp}(z_a)\overline{\varphi_{kp}(z_b)}}{1-e^{i(a-\bar{b})\sigma_{kp-1}}}
\]

Now we split the calculation into cases.

\subsection{Case 1: $\eitheta$ is not a resonance}\label{non}

We begin with several lemmas that will help us understand the necessary asymptotics of the orthonormal polynomials.

\begin{lemma}\label{atx1}
\begin{itemize}
\item[(a)]  If $\eitheta$ is on the interior of a band in the support of $\mu$ and $\eitheta$ is not a resonance, then $|\varphi_n(e^{i(\theta+c/n)})|=O(1)$ as $\nri$ uniformly for $c$ in compact subsets of $\bbC$.
\item[(b)]   If $\eitheta$ is at the edge of a band in the support of $\mu$ and $\eitheta$ is not a resonance, then $|\varphi_n(e^{i(\theta+c/n^2)})|=O(n)$ as $\nri$ uniformly for $c$ in compact subsets of $\bbC$.
\end{itemize}
\end{lemma}

\noindent\textit{Remark:}  The scaling of $n^{-1}$ and $n^{-2}$ are exactly what one would predict based on Bernstein's Inequality and Markov's Inequality (see \cite[Theorem 5.1.7]{BE} and \cite[Theorem 5.1.8]{BE} respectively).

\begin{proof}
In both cases, Theorem \ref{closed} shows that it suffices to show that the lemma holds when $n$ is a multiple of $p$.

\smallskip

(a)  By assumption, we can write
\[
\frac{\Delta(e^{i(\theta+c/n)})}{2}=x+\frac{t}{n}+O(n^{-2}),\qquad\qquad\nri
\]
for some $x\in(-1,1)$ and $t=t(c)\in\bbC$ that remains bounded as $c$ varies through any compact set in $\bbC$.  Since $|x|<1$, the $\arccos$ function is analytic in a neighborhood of $x$ and hence we can define
\[
w_n=\arccos\left(\frac{\Delta(e^{i(\theta+c/n)})}{2}\right)=\arccos(x)+\frac{\beta}{n}+O(n^{-2}),\qquad\nri
\]
for some $\beta=\beta(c)$ that remains bounded as $c$ varies through any compact set in $\bbC$.  If we let $\alpha=\arccos(x)$, then
\[
U_n\left(\frac{\Delta(e^{i(\theta+c/n)})}{2}\right)=U_n(\cos(w_n))=\frac{\sin((n+1)\alpha+\frac{n+1}{n}\beta+o(n^{-1}))}{\sin(\alpha+o(1))},\qquad\nri.
\]
Using basic trigonometric identities and the fact that $x$ is real with $\sin(x)\neq0$, we see that this expression is $O(1)$ as $\nri$.  The desired conclusion now follows from Theorem \ref{closed}.

\smallskip

(b) In this case, $|\Delta(\eitheta)|=2$, so without loss of generality, $\Delta(\eitheta)=2$.  Taylor expanding shows
\[
\frac{\Delta(e^{i(\theta+c/n^2)})}{2}=1-\frac{t}{n^2}+O(n^{-4}),\qquad\qquad \nri,
\]
for some complex number $t=t(c)$ that remains bounded as $c$ varies throughout any compact set in $\bbC$.  We can use \eqref{chebform2} to write
\begin{equation}\label{chebcos}
U_n\left(1-\frac{x^2}{2n^2}\right)= U_n(\cos(x/n))+o(n)= \frac{n\sin(x)}{x}+o(n),\qquad\qquad\nri.
\end{equation}
The desired conclusion follows from Theorem \ref{closed}.
\end{proof}

By applying Lemma \ref{atx1}, the Cauchy-Schwarz inequality, and Montel's Theorem, we obtain a proof of our next lemma.  Its importance is that it will allow us to prove only pointwise convergence, since the uniform convergence will follow automatically.

\begin{lemma}\label{normal}
Suppose $\eitheta\in\supp(\mu)$.  If $\eitheta$ is an interior point in a band in the support of $\mu$ and $\eitheta$ is not a resonance, then
\[
\left\{\frac{K_n(e^{i(\theta+a/n)},e^{i(\theta+b/n)};\mu)}{n}\right\}_{n\in\bbN}
\]
is a normal family on $\bbC^2$ in the variables $a$ and $b$.  If $\eitheta$ is a band edge and not a resonance, then
\[
\left\{\frac{K_n(e^{i(\theta+a/n^2)},e^{i(\theta+b/n^2)};\mu)}{n^3}\right\}_{n\in\bbN}
\]
is a normal family on $\bbC^2$ in the variables $a$ and $b$.
\end{lemma}

Now we proceed to calculate the asymptotics of the Christoffel-Darboux kernel.  By invoking Theorem \ref{closed}, we can rewrite the Christoffel Darboux formula as
\begin{align}
\nonumber K_{kp-1}(z_a,z_b;\mu)=&e^{i\frac{kp(a-\bar{b})\sigma_{kp-1}}{2}}\left(1-e^{i(a-\bar{b})\sigma_{kp-1}}\right)^{-1}\times\\
\nonumber&\quad\bigg(U_{k}\left(\frac{\Delta(z_a)}{2}\right)\overline{U_{k-1}\left(\frac{\Delta(z_b)}{2}\right)}\overline{\left(\frac{\varphi_p^*(z_b)-\varphi_p(z_b)}{2z_b^{p/2}}\right)}\\
\label{long}&\qquad+U_{k-1}\left(\frac{\Delta(z_a)}{2}\right)\overline{U_{k}\left(\frac{\Delta(z_b)}{2}\right)}\frac{\varphi_p^*(z_a)-\varphi_p(z_a)}{2z_a^{p/2}}\\
\nonumber&\qquad\frac{[\eta(z_a;-1)\overline{\eta(z_b;-1)}-\eta(z_a;1)\overline{\eta(z_b;1)}]U_{k-1}\left(\frac{\Delta(z_a)}{2}\right)\overline{U_{k-1}\left(\frac{\Delta(z_b)}{2}\right)}}{4z_a^{p/2}\bar{z}_b^{p/2}}\bigg)
\end{align}

To help us analyze this form of the expression, we provide the following lemma.

\begin{lemma}\label{cross}
With $z_a$ and $z_b$ defined as above, it holds that
\begin{align}\label{etas}
\eta(z_a;-1)\overline{\eta(z_b;-1)}-\eta(z_a;1)\overline{\eta(z_b;1)}=O(\sigma_n),\qquad\qquad\nri.
\end{align}
The $O$-estimate is uniform in $a$ and $b$ in compact subsets of $\bbC$.
\end{lemma}

\begin{proof}
If we expand the expressions for $\eta(z;\pm1)$, we find that the expression on the left-hand side of (\ref{etas}) simplifies to
\begin{equation}\label{pairs}
2((\varphi_p(z_a)-\varphi_p^*(z_a))(\overline{\psi_p(z_b)+\psi_p^*(z_b)})+(\overline{\varphi_p(z_b)-\varphi_p^*(z_b)})(\psi_p(z_a)+\psi_p^*(z_a)))
\end{equation}
If we consider the Taylor expansion around $\eitheta$ of each polynomial in this expression, then it suffices to show that all of the constant terms cancel.  The resulting expression will then be $O(|z_a-\eitheta|+|z_b-\eitheta|)=O(\sigma_n)$, which is what we want to show.

To this end, notice that in $\bbD$, it holds that $|\varphi_p(z)|<|\varphi_p^*(z)|$ and the reverse inequality holds outside the closed unit disk.  The same is true for $\psi_p$ and $\psi_p^*$, so we conclude that there are real numbers $\{t_j\}_{j=1}^p$ and $\{s_j\}_{j=1}^p$ such that
\[
\varphi_p(z)-\varphi_p^*(z)=\frac{1+\alpha_{p-1}}{r}\prod_{j=1}^p(z-e^{it_j}),\qquad \psi_p(z)+\psi_p^*(z)=\frac{1+\alpha_{p-1}}{r}\prod_{j=1}^p(z-e^{is_j})
\]
The constant term in the Taylor expansion of \eqref{pairs} is
\[
2((\varphi_p(\eitheta)-\varphi_p^*(\eitheta))(\overline{\psi_p(\eitheta)+\psi_p^*(\eitheta)})+(\overline{\varphi_p(\eitheta)-\varphi_p^*(\eitheta)})(\psi_p(\eitheta)+\psi_p^*(\eitheta)))
\]
We calculate
\begin{align*}
&(\varphi_p(\eitheta)-\varphi_p^*(\eitheta))(\overline{\psi_p(\eitheta)+\psi_p^*(\eitheta)})=\frac{|1+\alpha_{p-1}|^2}{r^2}\prod_{j=1}^p(\eitheta-e^{it_j})(e^{-i\theta}-e^{-is_j})\\
&\qquad\qquad=\frac{|1+\alpha_{p-1}|^2}{r^2}\prod_{j=1}^p(e^{-it_j}-e^{-i\theta})(e^{is_j}-\eitheta)\prod_{j=1}^pe^{it_j}\prod_{j=1}^pe^{-is_j}\\
&\qquad\qquad=\frac{|1+\alpha_{p-1}|^2}{r^2}\prod_{j=1}^p(e^{-i\theta}-e^{-it_j})(\eitheta-e^{is_j})\frac{r(\varphi_p(0)-\varphi_p^*(0))}{1+\alpha_{p-1}}\overline{\left(\frac{r(\psi_p(0)+\psi_p^*(0))}{\alpha_{p-1}+1}\right)}\\
&\qquad\qquad=\frac{|1+\alpha_{p-1}|^2}{r^2}\prod_{j=1}^p(e^{-i\theta}-e^{-it_j})(\eitheta-e^{is_j})\frac{-\bar{\alpha}_{p-1}-1}{1+\alpha_{p-1}}\cdot\frac{1+\alpha_{p-1}}{\bar{\alpha}_{p-1}+1}\\
&\qquad\qquad=-(\overline{\varphi_p(\eitheta)-\varphi_p^*(\eitheta)})(\psi_p(\eitheta)+\psi_p^*(\eitheta))
\end{align*}
so the constant term in (\ref{pairs}) is zero as desired.
\end{proof}

We now consider interior points and edge points separately.

\subsubsection{Case 1a: $\Delta(\eitheta)\in(-2,2)$}\label{case1a}  In the first case, we will assume that $\Delta(\eitheta)\in(-2,2)$ and $\sigma_n=1/n$.  By Taylor expanding $\frac{\varphi_p^*(z)-\varphi_p(z)}{2z^{p/2}}$ around $\eitheta$, we see that for $c=a,b$
\begin{equation}\label{phistar}
\frac{\varphi_p^*(z_c)-\varphi_p(z_c)}{2z_c^{p/2}}=\frac{\varphi_p^*(\eitheta)-\varphi_p(\eitheta)}{2e^{ip\theta/2}}+O(\sigma_n),
\end{equation}
where the error estimate is uniform for $c$ in compact subsets of $\bbC$.  Since $\frac{\varphi_p^*(\eitheta)-\varphi_p(\eitheta)}{2e^{ip\theta/2}}$ is purely imaginary and not zero (we used \cite[Theorem 11.3.1]{OPUC2} here), we can write (\ref{long}) as
\begin{align*}
&(1+o(1))e^{i\frac{(a-\bar{b})}{2}}\left(-\frac{i(a-\bar{b})}{kp-1}\right)^{-1}\times\\
&\quad\bigg(\left[U_{k}\left(\frac{\Delta(z_a)}{2}\right)\overline{U_{k-1}\left(\frac{\Delta(z_b)}{2}\right)}-U_{k-1}\left(\frac{\Delta(z_a)}{2}\right)\overline{U_{k}\left(\frac{\Delta(z_b)}{2}\right)}\right]\frac{\varphi_p(\eitheta)-\varphi_p^*(\eitheta)}{2e^{ip\theta/2}}(1+O(k^{-1}))\\
&\qquad\frac{[\eta(z_a;-1)\overline{\eta(z_b;-1)}-\eta(z_a;1)\overline{\eta(z_b;1)}]U_{k-1}\left(\frac{\Delta(z_a)}{2}\right)\overline{U_{k-1}\left(\frac{\Delta(z_b)}{2}\right)}}{4e^{i\frac{a-\bar{b}}{2k}}}\bigg)
\end{align*}
as $\kri$.  We can apply the calculations from Lemma \ref{atx1} to conclude that the last term in (\ref{long}) is $O(k^{-1})$ as $\kri$ (here we used the fact that $\Delta(\eitheta)\in(-2,2)$).  We conclude that
\begin{align}
\nonumber K_{kp-1}(z_a,z_b;\mu)&=(1+o(1))e^{i\frac{a-\bar{b}}{2}}\frac{kp}{i(\bar{b}-a)}\cdot\frac{\varphi_p(\eitheta)-\varphi_p^*(\eitheta)}{2e^{ip\theta/2}}\times\\
\label{kp}&\quad\left[U_{k}\left(\frac{\Delta(z_a)}{2}\right)\overline{U_{k-1}\left(\frac{\Delta(z_b)}{2}\right)}-U_{k-1}\left(\frac{\Delta(z_a)}{2}\right)\overline{U_{k}\left(\frac{\Delta(z_b)}{2}\right)}\right]
\end{align}
as $\kri$.  If $\nu^*$ is the measure of orthogonality for the polynomials $\{U_n\}_{n\geq0}$, then
\[
K_n(x,y;\nu^*)=\frac{\overline{U_{n}(y)}U_{n+1}(x)-U_{n}(x)\overline{U_{n+1}(y)}}{2(x-\bar{y})}
\]
(see \cite[Section 3]{SimonCD}).  Letting $x=\Delta(z_a)/2$ and $y=\Delta(z_b)/2$, we find
\begin{align*}
&K_{kp-1}(z_a,z_b;\mu)=\\
&\qquad(1+o(1))e^{i\frac{a-\bar{b}}{2}}\frac{kp(\varphi_p(\eitheta)-\varphi_p^*(\eitheta))}{2i(\bar{b}-a)e^{ip\theta/2}}(\Delta(z_a)-\overline{\Delta(z_b)})K_{k-1}\left(\frac{\Delta(z_a)}{2},\frac{\Delta(z_b)}{2};\nu^*\right)
\end{align*}

Let us consider the asymptotics of this expression.  Using the fact that $\Delta(\eitheta)$ is real and $W(\theta)$ is real and non-zero, a Taylor expansion shows that
\[
\lim_{\kri}\frac{kp(\Delta(z_a)-\overline{\Delta(z_b)})}{\bar{b}-a}=-W(\theta).
\]
We can also use \cite[Theorem 3.11.6]{Rice} to conclude that as $\kri$
\begin{align*}
K_{k-1}\left(\frac{\Delta(z_a)}{2},\frac{\Delta(z_b)}{2};\nu^*\right)&= K_{k-1}\left(\frac{\Delta(\eitheta)}{2},\frac{\Delta(\eitheta)}{2};\nu^*\right)\frac{\sin\frac{W(\theta)(\bar{b}-a)}{p\sqrt{4-\Delta(\eitheta)^2}}}{\frac{W(\theta)(\bar{b}-a)}{p\sqrt{4-\Delta(\eitheta)^2}}}+o(k)\\
&=\frac{2kp\,\sin\frac{W(\theta)(\bar{b}-a)}{p\sqrt{4-\Delta(\eitheta)^2}}}{W(\theta)(\bar{b}-a)\sqrt{4-\Delta(\eitheta)^2}}+o(k)
\end{align*}
where we used \cite[Equation 3.11.42]{Rice}.  Therefore,
\[
\lim_{\kri}\frac{K_{kp-1}(z_a,z_b;\mu)}{k}=e^{i\frac{a-\bar{b}}{2}}\frac{p(\varphi_p^*(\eitheta)-\varphi_p(\eitheta))\sin\frac{W(\theta)(a-\bar{b})}{p\sqrt{4-\Delta(\eitheta)^2}}}{ie^{ip\theta/2}(a-\bar{b})\sqrt{4-\Delta(\eitheta)^2}}
\]
By Lemma \ref{normal}, these asymptotics also hold when $a=\bar{b}$.  Taking a limit as $a-\bar{b}\rightarrow0$ shows
\[
\lim_{\kri}\frac{K_{kp-1}(\eitheta,\eitheta;\mu)}{k}=\frac{W(\theta)(\varphi_p^*(\eitheta)-\varphi_p(\eitheta))}{ie^{ip\theta/2}(4-\Delta(\eitheta)^2)}
\]
Thus,
\begin{align}\label{int1a}
\lim_{\kri}\frac{K_{kp-1}(z_a,z_b;\mu)}{K_{kp-1}(\eitheta,\eitheta;\mu)}=e^{i\frac{a-\bar{b}}{2}}\sinc\frac{W(\theta)(a-\bar{b})}{p\sqrt{4-\Delta(\eitheta)^2}}=e^{i\frac{a-\bar{b}}{2}}\sinc(V(\theta)(a-\bar{b}))
\end{align}
(compare with \cite[Equation 1.5]{LeviLub} and \cite[Theorem 1.1]{LubNg}).  Notice that all of the above estimates can be taken uniformly on any compact subset of $\Delta^{-1}((-2,2))$ because of the uniformity of the asymptotics of $K_n(x+a/n,x+b/n;\nu^*)$ for $x$ in compact subsets of $(-1,1)$.

\subsubsection{Case 1b: $\eitheta$ is a nonresonance band edge}\label{case1b}  Now we will consider the case when $\Delta(\eitheta)=2$, $\sigma_n=-1/n^2$, $\eitheta$ is at the edge of a band, and $\varphi_p(\eitheta)\neq\varphi_p^*(\eitheta)$.  According to \cite[Theorem 11.3.2iii]{OPUC2}, the weight of $\mu$ vanishes as a square root on the band edge near $\eitheta$.  As in the previous case, we need to calculate the asymptotics in the expression (\ref{long}), which we rewrite as
\begin{align*}
&(1+o(1))e^{i\frac{(a-\bar{b})k}{2(kp-1)^2}}\left(-\frac{i(a-\bar{b})}{(kp-1)^2}\right)^{-1}\times\\
&\quad\bigg(\left[U_{k}\left(\frac{\Delta(z_a)}{2}\right)\overline{U_{k-1}\left(\frac{\Delta(z_b)}{2}\right)}-U_{k-1}\left(\frac{\Delta(z_a)}{2}\right)\overline{U_{k}\left(\frac{\Delta(z_b)}{2}\right)}\right]\frac{\varphi_p(\eitheta)-\varphi_p^*(\eitheta)}{2e^{ip\theta/2}}(1+O(k^{-2}))\\
&\qquad\frac{[\eta(z_a;-1)\overline{\eta(z_b;-1)}-\eta(z_a;1)\overline{\eta(z_b;1)}]U_{k-1}\left(\frac{\Delta(z_a)}{2}\right)\overline{U_{k-1}\left(\frac{\Delta(z_b)}{2}\right)}}{4e^{i\frac{a-\bar{b}}{2k^2}}}\bigg)
\end{align*}
as $\kri$.  By applying Lemma \ref{cross} and the calculations in the proof of Lemma \ref{atx1}, we conclude that the last term in this expression is $O(1)$ as $\kri$.  It follows that
\begin{align}
\nonumber K_{kp-1}(z_a,z_b;\mu)&=(1+o(1))\frac{(kp)^2}{i(\bar{b}-a)}\cdot\frac{\varphi_p(\eitheta)-\varphi_p^*(\eitheta)}{2e^{ip\theta/2}}\times\\
\label{kp2}&\quad\left[(\Delta(z_a)-\overline{\Delta(z_b)})K_{k-1}\left(\frac{\Delta(z_a)}{2},\frac{\Delta(z_b)}{2};\nu^*\right)+O(1)\right]
\end{align}
as $\kri$, where $\nu^*$ is as in the previous case.

To calculate the asymptotics of this expression, we use \cite[Theorem 1.4]{Danka} to conclude that
\begin{align*}
K_{k-1}\left(\frac{\Delta(z_a)}{2},\frac{\Delta(z_b)}{2};\nu^*\right)&= K_{k-1}\left(1,1;\nu^*\right)\frac{\bbJ^*_{1/2}\left(\frac{W(\theta)a}{p^2},\frac{W(\theta)\bar{b}}{p^2}\right)}{\bbJ_{1/2}^*(0,0)}+o(k^3)\\
&=\frac{k(k+1)(2k+1)\bbJ^*_{1/2}\left(\frac{W(\theta)a}{p^2},\frac{W(\theta)\bar{b}}{p^2}\right)}{6\bbJ_{1/2}^*(0,0)}+o(k^3)
\end{align*}
A Taylor expansion shows $\Delta(z_a)-\overline{\Delta(z_b)}=O(k^{-2})$ as $\kri$, which means the $O(1)$ term in (\ref{kp2}) is negligible as $\kri$.  It also shows that
\[
\lim_{\kri}\frac{(kp)^2(\Delta(z_a)-\overline{\Delta(z_b)})}{\bar{b}-a}=W(\theta).
\]
Therefore,
\[
\lim_{\kri}\frac{K_{kp-1}(z_a,z_b;\mu)}{k^3}=\frac{(\varphi_p^*(\eitheta)-\varphi_p(\eitheta))\bbJ^*_{1/2}\left(\frac{W(\theta)a}{p^2},\frac{W(\theta)\bar{b}}{p^2}\right)}{6ie^{ip\theta/2}\bbJ_{1/2}^*(0,0)}
\]
By Lemma \ref{normal}, these asymptotics also hold when $a=\bar{b}$.  Taking a limit as $a,b\rightarrow0$ shows
\[
\lim_{\kri}\frac{K_{kp-1}(\eitheta,\eitheta;\mu)}{k^3}=\frac{(\varphi_p^*(\eitheta)-\varphi_p(\eitheta))}{6ie^{ip\theta/2}}
\]
Thus,
\begin{align}\label{int1b}
\lim_{\kri}\frac{K_{kp-1}(z_a,z_b;\mu)}{K_{kp-1}(\eitheta,\eitheta;\mu)}=\frac{\bbJ^*_{1/2}\left(\frac{W(\theta)a}{p^2},\frac{W(\theta)\bar{b}}{p^2}\right)}{\bbJ_{1/2}^*(0,0)}
\end{align}
as desired.

\subsection{Case 2: $\eitheta$ is a resonance}\label{res}

In this situation we will not rely on asymptotics of $K_n(\cdot,\cdot;\nu^*)$ and instead rely on more direct calculations.  We begin with a result that tells us how the second kind polynomials behave at a resonance.

\begin{lemma}\label{psiplus}
Suppose $\eitheta$ is a resonance and $\Delta(\eitheta)=2u$ with $u\in\{-1,1\}$.  Then
\begin{equation}\label{psis}
\frac{\psi_p(\eitheta)+\psi^*_p(\eitheta)}{2e^{ip\theta/2}}=u.
\end{equation}
\end{lemma}

\begin{proof}
We present the argument in the case $\Delta(\eitheta)=2$, the other case being very similar.  By Theorem \ref{closed}, we write
\begin{align*}
\varphi_{kp}(\eitheta)&=e^{ip\theta/2}\left(U_k(1)-\frac{\psi_p(\eitheta)+\psi^*_p(\eitheta)}{2e^{ip\theta/2}}U_{k-1}(1)\right)\\
&=e^{ip\theta/2}\left(k+1-\frac{\psi_p(\eitheta)+\psi^*_p(\eitheta)}{2e^{ip\theta/2}}k\right)
\end{align*}
If $\eitheta$ is on the interior of a band in $\supp(\mu)$, then Lemma \ref{normal} shows that
\begin{equation}\label{bigN}
\sum_{k=0}^{N}|\varphi_{kp}(\eitheta)|^2=O(N)
\end{equation}
as $N\rightarrow\infty$.  If $\eitheta$ is at a band edge, then by \cite[Theorem 11.3.2iv]{OPUC2}, we know that the weight of $\mu$ grows like the inverse of the square root of the distance near $\eitheta$ and hence by \cite[Theorem 1.1]{DT} we also know that \eqref{bigN} holds.  By our above formula, the only way this can happen is if \eqref{psis} holds.
\end{proof}

We will also need the following lemma, which is an analog of Lemma \ref{atx1}.

\begin{lemma}\label{atx2}
\begin{itemize}
\item[(a)]  If $\eitheta$ is on the interior of a band in the support of $\mu$ and $\eitheta$ is a resonance, then $|\varphi_n(e^{i(\theta+c/n)})|=O(1)$ as $\nri$ uniformly for $c$ in compact subsets of $\bbC$.
\item[(b)]   If $\eitheta$ is at the edge of a band in the support of $\mu$ and $\eitheta$ is a resonance, then $|\varphi_n(e^{i(\theta+c/n^2)})|=O(1)$ as $\nri$ uniformly for $c$ in compact subsets of $\bbC$.
\end{itemize}
\end{lemma}

\begin{proof}
In both cases we may assume without loss of generality that $\Delta(\eitheta)=2$.  Theorem \ref{closed} shows that it suffices to show that the lemma holds when $n$ is a multiple of $p$.

\smallskip

(a)  We use Theorem \ref{closed} and Lemma \ref{psiplus} to deduce that
\begin{align*}
&|\varphi_{kp}(e^{i(\theta+c/(kp))})|=O\left|U_k\left(\frac{\Delta(z_c)}{2}\right)-U_{k-1}\left(\frac{\Delta(z_c)}{2}\right)\right|,
\end{align*}
where $z_c=e^{i(\theta+c/(kp))}$.  A resonance on the interior of a band edge must be a closed gap and as in the proof of Lemma \ref{atx1}, we may assume $\Delta(\eitheta)=2$.  This implies $\Delta'(\eitheta)=0$ and $W'(\theta)<0$ so we can write
\[
\frac{\Delta(z_c)}{2}=1+\frac{W'(\theta)c^2}{4(kp-1)^2}+O(k^{-3})=\cos\left(\frac{c\sqrt{|W'(\theta)|}}{\sqrt{2}(kp-1)}\right)+O(k^{-3}).
\]
as $\kri$.  Thus, if we define $t_k$ so that $\cos(t_k)=\Delta(z_c)/2$, it holds that $t_k=\frac{c\sqrt{|W'(\theta)|}}{\sqrt{2}kp}+o(k^{-1})$ as $\kri$.  Then the formula \eqref{chebform2} implies
\[
U_k\left(\frac{\Delta(z_c)}{2}\right)-U_{k-1}\left(\frac{\Delta(z_c)}{2}\right)=\frac{\sin((k+1)t_k)-\sin(kt_k)}{\sin(t_k)}=\frac{\sin(kt_k)}{\sin(t_k)}(\cos(t_k)-1)+\cos(kt_k)=O(1)
\]
as $\kri$.

\smallskip

(b) The proof is very similar to part (a) except now we set $z_c=e^{i(\theta+c/(kp)^2)}$ and again choose $t_k\in\bbC$ so that $\cos(t_k)=\Delta(z_c)/2$.  Since $\Delta(\eitheta)=2$, we see that $t_k=O(k^{-1})$ as $\kri$.  By using \eqref{chebform2} as in part (a) we find
\[
U_k(\cos(t_k))-U_{k-1}(\cos(t_k))=\frac{\sin((k+1)t_k)-\sin(kt_k)}{\sin(t_k)}=O(1)
\]
as $\kri$.
\end{proof}

By combining Lemma \ref{psiplus} and Lemma \ref{atx2} we obtain the following analog of\ Lemma \ref{normal} to the case of resonances.

\begin{lemma}\label{normal2}
Suppose $\eitheta\in\supp(\mu)$.  If $\eitheta$ is a resonance, then
\[
\left\{\frac{K_n(e^{i(\theta+a/n)},e^{i(\theta+b/n)};\mu)}{n}\right\}_{n\in\bbN}
\]
is a normal family on $\bbC^2$ in the variables $a$ and $b$.
\end{lemma}

Now we proceed as in Case 1 to consider interior points and edge points of the bands separately.

\subsubsection{Case 2a: $\eitheta$ is a closed gap}\label{case2a}  Now we will consider the case when $\Delta(\eitheta)\in\{-2,2\}$ and $W(\theta)=0$.  In other words, the point $\eitheta$ represents a closed gap, so we will let $\sigma_n=1/n$.  Due to the symmetry of the problem, we will confine our attention to the case $\Delta(\eitheta)=2$.  As in the proof of Lemma \ref{atx2}, for $c=a,b$
\[
\frac{\Delta(z_c)}{2}=1+\frac{W'(\theta)c^2}{4(kp-1)^2}+O(k^{-3})
\]
as $\kri$.

Define
\begin{align*}
f_1(z):=\frac{\varphi_p^*(z)-\varphi_p(z)}{2z^{p/2}},&\qquad\qquad g_1(t):=f_1(e^{it})\\
f_2(z):=\frac{\psi_p^*(z)+\psi_p(z)}{2z^{p/2}},&\qquad\qquad g_2(t):=f_2(e^{it})\
\end{align*}
We know that $g_1:\bbR\rightarrow i\bbR$ and $g_2:\bbR\rightarrow\bbR$, so the same must be true for $g_1'$ and $g_2'$ respectively.  Using these formulas, we can use \eqref{chebcos} to estimate
\begin{align*}
&\frac{\varphi_{kp}(z_a)}{e^{i\theta kp/2}}=e^{ikpa/(2kp-2)}\bigg[U_{k}\left(1+\frac{W'(\theta)a^2}{4(kp-1)^2}+O(k^{-3})\right)\\
&\qquad\qquad-\left(g_1\left(\theta+\frac{a}{kp-1}\right)+g_2\left(\theta+\frac{a}{kp-1}\right)\right)U_{k-1}\left(1+\frac{W'(\theta)a^2}{4(kp-1)^2}+O(k^{-3})\right)\bigg]\\
&= e^{ia/2}\left[U_{k}\left(\cos\left(\frac{a\sqrt{|W'(\theta)|}}{\sqrt{2}(kp-1)}\right)\right)-U_{k-1}\left(\cos\left(\frac{a\sqrt{|W'(\theta)|}}{\sqrt{2}(kp-1)}\right)\right)\left(1+\frac{(g_1'(\theta)+g_2'(\theta))a}{kp-1}\right)\right]+o(1)\\
&= e^{ia/2}\left(\cos\left(\frac{a\sqrt{|W'(\theta)|}}{p\sqrt{2}}\right)-\frac{\sin\left(\frac{a\sqrt{|W'(\theta)|}}{p\sqrt{2}}\right)}{\sin\left(\frac{a\sqrt{|W'(\theta)|}}{\sqrt{2}(kp-1)}\right)}\cdot\frac{(g_1'(\theta)+g_2'(\theta))a}{kp-1}\right)+o(1)
\end{align*}
as $\kri$.  A similar analysis of $\varphi_{kp}^*(z_a)$ shows that as $\kri$
\begin{align*}
\frac{\varphi_{kp}(z_a)}{e^{i\theta kp/2}}&= e^{ia/2}\left(\cos\left(\frac{a\sqrt{|W'(\theta)|}}{p\sqrt{2}}\right)-\frac{\sqrt{2}\sin\left(\frac{a\sqrt{|W'(\theta)|}}{p\sqrt{2}}\right)(g_1'(\theta)+g_2'(\theta))}{\sqrt{|W'(\theta)|}}\right)+o(1)\\
\frac{\varphi_{kp}^*(z_a)}{e^{i\theta kp/2}}&= e^{ia/2}\left(\cos\left(\frac{a\sqrt{|W'(\theta)|}}{p\sqrt{2}}\right)-\frac{\sqrt{2}\sin\left(\frac{a\sqrt{|W'(\theta)|}}{p\sqrt{2}}\right)(g_2'(\theta)-g_1'(\theta))}{\sqrt{|W'(\theta)|}}\right)+o(1)
\end{align*}
Thus, if $a\neq\bar{b}$, then the Christoffel Darboux formula implies
\begin{align*}
&\lim_{\kri}\frac{K_{kp-1}(z_a,z_b;\mu)}{k}\\
&\,=\frac{2\sqrt{2}g_1'(\theta)pe^{i\frac{a-\bar{b}}{2}}}{i(\bar{b}-a)\sqrt{|W'(\theta)|}}\left(\cos\left(\frac{\bar{b}\sqrt{|W'(\theta)|}}{p\sqrt{2}}\right)\sin\left(\frac{a\sqrt{|W'(\theta)|}}{p\sqrt{2}}\right)-\cos\left(\frac{a\sqrt{|W'(\theta)|}}{p\sqrt{2}}\right)\sin\left(\frac{\bar{b}\sqrt{|W'(\theta)|}}{p\sqrt{2}}\right)\right)\\
&\,=\frac{2\sqrt{2}g_1'(\theta)pe^{i\frac{a-\bar{b}}{2}}}{i(\bar{b}-a)\sqrt{|W'(\theta)|}}\sin\left(\frac{(a-\bar{b})\sqrt{|W'(\theta)|}}{p\sqrt{2}}\right)
\end{align*}
By Lemma \ref{normal2}, these same asymptotics hold when $a=\bar{b}$, so taking $a-\bar{b}\rightarrow0$ shows
\begin{align*}
&\lim_{\kri}\frac{K_{kp-1}(\eitheta,\eitheta;\mu)}{k}=2ig_1'(\theta)
\end{align*}
In particular, we see that $g_1'(\theta)\neq0$.  We conclude that
\begin{align*}
\lim_{\kri}\frac{K_{kp-1}(z_a,z_b;\mu)}{K_{kp-1}(\eitheta,\eitheta;\mu)}=e^{i\frac{a-\bar{b}}{2}}\frac{\sin\left(\frac{(a-\bar{b})\sqrt{|W'(\theta)|}}{p\sqrt{2}}\right)}{\frac{(a-\bar{b})\sqrt{|W'(\theta)|}}{p\sqrt{2}}}=e^{i\frac{a-\bar{b}}{2}}\sinc\left(\frac{(a-\bar{b})\sqrt{|W'(\theta)|}}{p\sqrt{2}}\right)
\end{align*}
An application of L'H\^{o}pital's Rule to the formula for $V(\theta)$ shows that we can rewrite this as
\begin{align}\label{int2a}
\lim_{\kri}\frac{K_{kp-1}(z_a,z_b;\mu)}{K_{kp-1}(\eitheta,\eitheta;\mu)}=e^{i\frac{a-\bar{b}}{2}}\sinc\left(V(\theta)(a-\bar{b})\right)
\end{align}
as in (\ref{int1a}).

\subsubsection{Case 2b: $\eitheta$ is a band edge}\label{case2b}

As above, we will consider only the case $\Delta(\eitheta)=2$.  We will carry out an analysis similar to that in Case 2a but with $\sigma_n=-1/n^2$.  If we define $t_k$ so that $\cos\left(\frac{t_k}{k}\right)=\Delta(z_a)/2$, then
\begin{align*}
\frac{\varphi_{kp}(z_a)}{e^{i\theta kp/2}}&=\left(1-\frac{iakp}{2(kp-1)^2}+O(k^{-2})\right)\\
&\qquad\times\left[\frac{\sin(t_k)}{\sin\left(\frac{t_k}{k}\right)}(\cos\left(\frac{t_k}{k}\right)-1)+\cos(t_k)+\frac{\sin(t_k)a(g_2'(\theta)+g_1'(\theta))}{\sin\left(\frac{t_k}{k}\right)(kp-1)^2}+O(k^{-3})\right]\\
\frac{\varphi_{kp}^*(z_a)}{e^{i\theta kp/2}}&=\left(1-\frac{iakp}{2(kp-1)^2}+O(k^{-2})\right)\\
&\qquad\times\left[\frac{\sin(t_k)}{\sin\left(\frac{t_k}{k}\right)}(\cos\left(\frac{t_k}{k}\right)-1)+\cos(t_k)+\frac{\sin(t_k)a(g_2'(\theta)-g_1'(\theta))}{\sin\left(\frac{t_k}{k}\right)(kp-1)^2}+O(k^{-3})\right]
\end{align*}
Notice that
\[
\sum_{m=0}^{\infty}(-1)^m\frac{t_k^{2m}}{k^{2m}(2m)!}=1-\frac{W(\theta)a}{2(kp-1)^2}+O(k^{-4}),\qquad\qquad\kri
\]
and so
\[
\frac{W(\theta)a}{(kp-1)^2}=\frac{t_k^2}{k^2}+O(k^{-4}),\qquad\qquad\kri,
\]
which means $t_k=\frac{\sqrt{aW(\theta)}}{p}+O(k^{-1})$ as $\kri$ for an appropriate choice of the square root.  Therefore,
\begin{align*}
\frac{\varphi_{kp}(z_a)}{e^{i\theta kp/2}}&=\cos\left(\frac{\sqrt{aW(\theta)}}{p}\right)+O(k^{-2})\\
&\qquad+\frac{1}{k}\left[\frac{-ia\cos\left(\frac{\sqrt{aW(\theta)}}{p}\right)}{2p}-\frac{\sqrt{aW(\theta)}}{p}\sin\left(\frac{\sqrt{aW(\theta)}}{p}\right)\left(\frac{1}{2}-\frac{g_1'(\theta)+g_2'(\theta)}{W(\theta)}\right)\right]\\
\frac{\varphi_{kp}^*(z_a)}{e^{i\theta kp/2}}&=\cos\left(\frac{\sqrt{aW(\theta)}}{p}\right)+O(k^{-2})\\
&\qquad+\frac{1}{k}\left[\frac{-ia\cos\left(\frac{\sqrt{aW(\theta)}}{p}\right)}{2p}-\frac{\sqrt{aW(\theta)}}{p}\sin\left(\frac{\sqrt{aW(\theta)}}{p}\right)\left(\frac{1}{2}-\frac{g_2'(\theta)-g_1'(\theta)}{W(\theta)}\right)\right]\\
\end{align*}
as $\kri$.  A similar formula holds at $z_b$ so we can now use the Christoffel-Darboux formula to calculate
\begin{align*}
&\lim_{\kri}\frac{K_{kp-1}(z_a,z_b;\mu)}{k}\\
&\,=\frac{2pg_1'(\theta)}{i(\bar{b}-a)W(\theta)}\\
&\quad\times\left(\sqrt{aW(\theta)}\cos\left(\frac{\sqrt{\overline{b}W(\theta)}}{p}\right)\sin\left(\frac{\sqrt{aW(\theta)}}{p}\right)-\sqrt{\bar{b}W(\theta)}\cos\left(\frac{\sqrt{aW(\theta)}}{p}\right)\sin\left(\frac{\sqrt{\overline{b}W(\theta)}}{p}\right)\right)\\
&\,=2p\pi ig_1'(\theta)\bbJ_{-1/2}^*\left(\frac{W(\theta)a}{p^2},\frac{W(\theta)\bar{b}}{p^2}\right)
\end{align*}
where we used the fact that
\[
J_{-1/2}(z)=\frac{\sqrt{2}\cos(z)}{\sqrt{\pi z}}
\]
(see \cite[page 17]{Korenev}).  By Lemma \ref{normal2}, these same asymptotics hold when $a=\bar{b}$.  Setting $a=b=0$, we find
\begin{align}\label{int2b}
\lim_{\kri}\frac{K_{kp-1}(z_a,z_b;\mu)}{K_{kp-1}(\eitheta,\eitheta;\mu)}=\frac{\bbJ_{-1/2}^*\left(\frac{W(\theta)a}{p^2},\frac{W(\theta)\bar{b}}{p^2}\right)}{\bbJ_{-1/2}^*\left(0,0\right)}
\end{align}
as desired.

\medskip

Combining (\ref{int1a}), (\ref{int1b}), (\ref{int2a}), and (\ref{int2b}) proves Theorem \ref{univthm} in the case $n=kp-1$ and $\mu^*=\mu$.  To pass from the subsequence $\{kp-1\}_{k\in\bbN}$ to the general case, we need the following lemma.

\begin{lemma}\label{reduce}
Suppose $\mu$ is as in the statement of Theorem \ref{univthm}.  Fix $\eitheta$ in a band of $\supp(\mu)$.  If $\Delta(\eitheta)\in(-2,2)$, set $\sigma_n=1/n$, while if $\Delta(\eitheta)\in\{-2,2\}$, set $\sigma_n=-1/n^2$.  Then uniformly for $a$ and $b$ in compact subsets of $\bbC$ and $s\in\{1,\ldots,p-1\}$ it holds that
\[
\lim_{\kri}\left|\frac{K_{kp-1+s}(e^{i(\theta+a\sigma_{kp-1+s})},e^{i(\theta+b\sigma_{kp-1+s})};\mu)}{K_{kp-1+s}(e^{i\theta},e^{i\theta};\mu)}- \frac{K_{kp-1}(e^{i(\theta+a\sigma_{kp-1})},e^{i(\theta+b\sigma_{kp-1})};\mu)}{K_{kp-1}(e^{i\theta},e^{i\theta};\mu)}\right|=0
\]
as $\kri$.
\end{lemma}

\begin{proof}
First we will show that
\begin{align}\label{srat}
\lim_{\kri}\frac{K_{kp-1+s}(e^{i\theta},e^{i\theta};\mu)}{K_{kp-1}(e^{i\theta},e^{i\theta};\mu)}=1.
\end{align}
Then we will show that
\begin{align}\label{sumrat}
\lim_{\kri}\frac{\sum_{j=kp}^{kp+s-1}\varphi_j(e^{i(\theta+a\sigma_{kp+s-1})})\overline{\varphi_j(e^{i(\theta+b\sigma_{kp+s-1})})}}{K_{kp-1}(e^{i\theta},e^{i\theta};\mu)}=0.
\end{align}
uniformly for $a$ and $b$ in compact subsets of $\bbC$ and uniformly in $s\in\{1,\ldots,p-1\}$.  These two limits and the calculations of this section combine to give us the desired result.  Notice that (\ref{srat}) is (\ref{sumrat}) in the case $a=b=0$, so it suffices to prove (\ref{sumrat}).

Since $\mu(\{\eitheta\})=0$, we know that $K_{kp-1}(\eitheta,\eitheta;\mu)\rightarrow\infty$ as $\kri$.  If $\Delta(\eitheta)\in(-2,2)$, then Lemmas \ref{atx1} and \ref{atx2} shows that $\{|\varphi_j(e^{i(\theta+c\sigma_{j})})|\}_{j\in\bbN}$ is bounded uniformly in $j\in\bbN$ and uniformly in $c$ in compact subsets of $\bbC$.  This implies (\ref{sumrat}).  The same reasoning applies if $|\Delta(\eitheta)|=2$ and $\eitheta$ is a resonance.  If $|\Delta(\eitheta)|=2$ and $\eitheta$ is not a resonance, then Lemma \ref{atx1} shows that $|\varphi_j(e^{i(\theta+c\sigma_{j})})|=O(j)$ as $j\rightarrow\infty$ uniformly in $c$ in compact subsets of $\bbC$, while $K_{kp-1}(\eitheta,\eitheta;\mu)$ grows like a constant times $k^3$ as $\kri$ (see also \cite[Theorem 1.1]{DT}).  Taken together, this implies (\ref{sumrat}) in this case.
\end{proof}

Lemma \ref{reduce} and the calculations preceding it give us a complete proof of Theorem \ref{univthm} in the case $\mu^*=\mu$.  The general case now follows standard perturbative methods as in \cite{LeviLub,LubNg,SimUniv,SimGeron}.  Due to the well-known method of proof, we provide only a sketch of the ideas.  The main idea is to rely on a result of Bourgade, which is \cite[Theorem 3.10]{Bourgade} and tells us that under the assumptions of Theorem \ref{univthm} it holds that
\[
\lim_{\nri}\left|\frac{w(\theta)K_{n}(e^{i(\theta+a\sigma_n)},e^{i(\theta+b\sigma_n)};\mu^*)-K_{n}(e^{i(\theta+a\sigma_n)},e^{i(\theta+b\sigma_n)};\mu)}{K_{n}(e^{i(\theta+a\sigma_n)},e^{i(\theta+a\sigma_n)};\mu^*)}\right|=0
\]
uniformly for $a$ and $b$ in appropriate compact intervals in the real line.  Some technicalities remain to show that this is the desired conclusion.

The key remaining part of the proof of Theorem \ref{univthm} is to show that
\[
\lim_{\nri}\frac{K_n(e^{i(\theta-a\sigma_n)},e^{i(\theta-a\sigma_n)};\mu)}{K_n(e^{i(\theta-a\sigma_n)},e^{i(\theta-a\sigma_n)};\mu^*)}=w(\theta)
\]
and the convergence is uniform for $a$ in compact subsets of $\bbC$.  Proving this is where we use the fact that the measure $\mu^*$ is regular and involves interpreting the diagonal reproducing kernel in terms of Christoffel functions and a localization argument.  The details appear in many places, such as \cite{LeviLub,MNT,SimUniv}, so we omit the lengthy calculations.  Due to the possible presence of mass points in the gaps of $\supp(\mu_{ac})$, we mention that the Erd\H{o}s-Tur\'{a}n criterion (see \cite[page 101]{StaTo}) and \cite[Theorem 11.3.2]{OPUC2} imply that not only is $\mu^*$ regular, but so is its restriction to $\supp(\mu_{ac})$.  Thus \cite[Theorem 3.2.3v]{StaTo} is applicable to this restricted measure.

Part of the calculations that allow us to complete the proof of Theorem \ref{univthm} require knowing that the limiting kernel is strictly positive along the diagonal.  This is given to us by the following result, which is an analog of \cite[Proposition 4.2]{SimGeron}.

\begin{prop}\label{jminus}
The function $\bbJ_{-1/2}^*(t,\bar{t})$ is non-vanishing as a function of $t\in\bbC$.
\end{prop}

\begin{proof}
First note that for $t\in\bbR$ it holds that
\[
\bbJ_{-1/2}^*(t,t)=\frac{1}{4\pi}\left(2+\frac{\sin(2\sqrt{t})}{\sqrt{t}}\right)>0,
\]
which verifies the claim in this case.  By using the formula for $J_{-1/2}$ mentioned earlier, one can check that if $a\neq b$
\[
\bbJ_{-1/2}^*(a,b)=\frac{\sqrt{a}\sin(\sqrt{a})\cos(\sqrt{b})-\sqrt{b}\sin(\sqrt{b})\cos(\sqrt{a})}{\pi(a-b)}
\]
If $\bbJ_{-1/2}^*(t,\bar{t})=0$ for $t\not\in\bbR$, then for $\theta$ as in Case 2b above, we would have
\[
\lim_{\nri}\frac{K_n(e^{i(\theta-p^2t/(W(\theta)n^2))},e^{i(\theta-p^2t/(W(\theta)n^2))},\mu)}{K_n(\eitheta,\eitheta,\mu)}=0.
\]
By the Cauchy-Schwarz inequality, this would imply
\[
\lim_{\nri}\frac{K_n(e^{i(\theta-p^2t/(W(\theta)n^2))},\eitheta,\mu)}{K_n(\eitheta,\eitheta,\mu)}=0.
\]
However, we also know this limit is equal to $\bbJ_{-1/2}^*(t,0)/\bbJ_{-1/2}^*(0,0)=\sin(\sqrt{t})/\sqrt{t}$.  This implies $t\in\bbR$, which gives us a contradiction.
\end{proof}

With these tools in hand, one will have no trouble adapting the method used in \cite{LeviLub,LubNg,SimUniv,SimGeron} to complete the proof of Theorem \ref{univthm}.

\section{Simple Applications of Theorem \ref{closed}}\label{simple}

In this section, we will use Theorem \ref{closed} to find short proofs of some formulas and results, some of which are new.  

\smallskip

As a first application, we recall that in \cite[Section 5.1]{dJP} a new derivation of \cite[Theorem 2.1]{Simon3} is given.  Using Theorem \ref{closed}, we can provide a new proof of the unit circle version, which was originally given as \cite[Theorem 3.1]{Simon3}.

\begin{corollary}\label{new31}
Under the assumptions of Theorem \ref{closed}, the zeros of $\Phi_{kp}(z)-\Phi_{kp}^*(z)$ are the resonances and the preimages of the zeros of $U_{k-1}(x)$ under the map $\Delta(z)/2$.
\end{corollary}

\noindent\textit{Remark.} Recall that resonances are the zeros of $\Phi_p(z)-\Phi_p^*(z)$.

\smallskip

\noindent\textit{Remark.}  Since $U_n$ is either odd or even (depending on the parity of $n$), we do not need to assume that $p$ is even in Corollary \ref{new31}.

\smallskip

As a second application of Theorem \ref{closed} and inspired by the work in \cite{Cantero}, let us consider the generating function of this sequence of orthonormal polynomials.  To state our results more succinctly, we define
\begin{align*}
g(z,t)&:=(\varphi_p(z)-\varphi_p^*(z))\sum_{s=0}^{p-1}t^s\psi_s(z)-(\psi_p(z)+\psi_p^*(z))\sum_{s=0}^{p-1}t^s\varphi_s(z)\\
\nu(z,t;\lambda)&:=\sum_{s=0}^{p-1}t^s(\varphi_s(z)+\lambda\psi_s(z))
\end{align*}

\begin{corollary}\label{genfunc}
The polynomials $\{\varphi_n(z)\}_{n=0}^{\infty}$ satisfy
\begin{align*}
\sum_{n=0}^{\infty}\varphi_n(z)t^n&=\frac{2\nu(z,t;0)+t^pg(z,t)}{2(1-\Delta(z)z^{p/2}t^p+z^pt^{2p})}\\
\end{align*}
whenever this series converges.
\end{corollary}

\begin{proof}
Notice that Theorem \ref{closed} implies
\begin{align}\label{uform}
\sum_{n=0}^{\infty}\varphi_n(z)t^n&=\frac{\nu(z,t;1)}{2}\sum_{k=0}^{\infty}t^{kp}\varphi_{kp}(z)+\frac{\nu(z,t;-1)}{2}\sum_{k=0}^{\infty}t^{kp}\varphi_{kp}^*(z)
\end{align}
Now we again use Theorem \ref{closed} and the fact that
\[
\sum_{n=0}^{\infty}U_n(x)t^n=\frac{1}{1-2xt+t^2}
\]
whenever this series converges (see \cite[Equation 4.5.23]{Ibook}).  This shows
\[
\sum_{k=0}^{\infty}t^{kp}\varphi_{kp}(z)=\frac{2+\eta(z;1)t^p}{2-2\Delta(z)z^{p/2}t^p+2z^pt^{2p}}
\]
and
\[
\sum_{k=0}^{\infty}t^{kp}\varphi_{kp}^*(z)=\frac{2+\eta(z;-1)t^p}{2-2\Delta(z)z^{p/2}t^p+2z^pt^{2p}}
\]
Now we plug this into (\ref{uform}) and use the fact that $\nu(z,t;1)+\nu(z,t;-1)=2\nu(z,t;0)$ and the definition of $g(z,t)$ to deduce the desired formula.
\end{proof}

We can also use Theorem \ref{closed} to derive identities satisfied by the Chebyshev polynomials $U_n$.  Indeed, a sequence of periodic Verblunsky coefficients with period $p$ is also a sequence of periodic Verblunsky coefficients with period $np$ for any $n\in\bbN$.  Therefore, we can use the formulas in Theorem \ref{closed} with different choices of periods to derive identities for the polynomials $\{U_n\}_{n\geq0}$.

\begin{corollary}\label{newu}
Let $\{\alpha_n\}_{n=0}^{\infty}$ be a sequence of periodic Verblunsky coefficients with period $p$.  Let $\Delta_j(z)$ be defined as $\Delta(z)$ in (\ref{deltadef}) but with each $p$ replaced by $j$ and similarly define $\eta_j(z;\sigma)$ as in Section \ref{periodic} but with each $p$ replaced by $j$.  Then for any $k,m\in\bbN$ it holds that
\[
U_k\left(\frac{\Delta_{mp}(z)}{2}\right)+\frac{\eta_{mp}(z;1)}{2z^{mp/2}}U_{k-1}\left(\frac{\Delta_{mp}(z)}{2}\right)=U_{mk}\left(\frac{\Delta_{p}(z)}{2}\right)+\frac{\eta_{p}(z;1)}{2z^{p/2}}U_{mk-1}\left(\frac{\Delta_{p}(z)}{2}\right)
\]
\end{corollary}

For example, if we consider the sequence $\{0,\alpha,0,\alpha,0,\ldots\}$, then applying Corollary \ref{newu} with $m=2$ shows that for any $k\in\bbN$ and any $\alpha\in\bbD$ it holds that
\begin{align*}
&U_k\left(\frac{z^4+2|\alpha|^2z^2+1}{2z^2(1-|\alpha|^2)}\right)-\frac{2(z^2(\,|\alpha|^2+\bar{\alpha})+\bar{\alpha}+1)}{2z^{2}(1-|\alpha|^2)}U_{k-1}\left(\frac{z^4+2|\alpha|^2z^2+1}{2z^2(1-|\alpha|^2)}\right)\\
&\qquad\qquad=U_{2k}\left(\frac{z^2+1}{2z\sqrt{1-|\alpha|^2}}\right)-\frac{1+\bar{\alpha}}{z\sqrt{1-|\alpha|^2}}U_{2k-1}\left(\frac{z^2+1}{2z\sqrt{1-|\alpha|^2}}\right)
\end{align*}

As an additional application of Theorem \ref{closed}, we can derive a formula for the Carath\'{e}odory function of the measure of orthogonality.  Such an expression is given in \cite[Equation 11.3.15]{OPUC2} and \cite[Section 2]{PehSt}, but we will give a different formula.   By \cite[Theorem 1]{SimRat}, we know that
\begin{equation}\label{urat}
\lim_{\nri}\frac{U_{n+1}(x)}{U_n(x)}=x+\sqrt{x^2-1},\qquad x\not\in[-1,1].
\end{equation}
Therefore, we have
\begin{align}
\nonumber F(z)=\lim_{k\rightarrow\infty}\frac{\psi_{kp}^*(z)}{\varphi_{kp}^*(z)}&=\frac{2z^{p/2}\Gamma_+(z)+\psi_p^*(z)-\psi_p(z)-\varphi_p(z)-\varphi_p^*(z)}{2z^{p/2}\Gamma_+(z)+\varphi_p^*(z)-\varphi_p(z)-\psi_p(z)-\psi_p^*(z)}\\
\label{fform}&=1+\frac{2(\psi_p^*(z)-\varphi_p^*(z))}{2z^{p/2}\Gamma_+(z)+\varphi_p^*(z)-\varphi_p(z)-\psi_p(z)-\psi_p^*(z)}
\end{align}

We can also use Theorem \ref{closed} to find the Schur function for the measure of orthogonality.  To do so, we follow the method of \cite{SimGeron} and first find convenient formulas for the Wall polynomials for this measure.  Recall that the Wall polynomials are a pair of polynomial sequences $\{A_n,B_n\}_{n=0}^{\infty}$ such that $A_n/B_n$ converges to the Schur function $f$ associated with $\mu$ uniformly on compact subsets of the unit disk $\bbD$.  The Pint\'{e}r-Nevai formulas (see \cite[Theorem 3.2.10]{OPUC1} or \cite{PN}) tell us that
\begin{align*}
A_n(z)&=\frac{\Psi_{n+1}^*(z;\mu)-\Phi_{n+1}^*(z;\mu)}{2z}\\
B_n(z)&=\frac{\Psi_{n+1}^*(z;\mu)+\Phi_{n+1}^*(z;\mu)}{2}
\end{align*}
Since we are free to take the index to infinity through any subsequence, we will only write explicit formulas for the Wall polynomials for indices in a certain subsequence.  Formulas for the other indices can be obtained easily.

\begin{corollary}\label{wall}
For all $k\in\bbN$, the Wall polynomials $A_{kp-1}$ and $B_{kp-1}$ for the measure $\mu$ are given by
\begin{align*}
A_{kp-1}(z)&=\frac{r^{k}z^{(k-1)p/2-1}U_{k-1}(\Delta(z)/2)}{2}(\psi_p^*(z)-\varphi_p^*(z))\\
B_{kp-1}(z)&=r^{k}z^{kp/2}\left[U_k(\Delta(z)/2)-\frac{U_{k-1}(\Delta(z)/2)}{2z^{p/2}}(\psi_p(z)+\varphi_p(z))\right]
\end{align*}
\end{corollary}

\begin{proof}
This is an immediate consequence of Theorem \ref{closed} and the Pint\'{e}r-Nevai formulas.
\end{proof}

We can apply Corollary \ref{wall} and send $k\rightarrow\infty$ to find the Schur function for the measure $\mu$.  If we apply Corollary \ref{wall} and \eqref{urat} with $x=\Delta(z)/2$ we conclude
\[
f(z)=\lim_{k\rightarrow\infty}\frac{A_{kp-1}(z)}{B_{kp-1}(z)}=\frac{\psi_p^*-\varphi_p^*}{2\Gamma_+(z)z^{p/2+1}-z(\psi_p(z)+\varphi_p(z))},\qquad\qquad |z|<1
\]

\smallskip

As a final application, we note that according to the results in \cite{SimShift}, each of the following limits should exist for $z$ in the appropriate domain:
\[
\lim_{k\rightarrow0}\frac{\varphi_{kp+s}(z)}{\varphi_{kp+s+1}(z)},\qquad\qquad s\in\{0,\ldots,p-1\}
\]
Another application of Theorem \ref{closed} (or Theorem \ref{jdef}) allows us to make these limits explicit (see also \cite[Corollary 3.5]{PehSt2}).



\end{document}